\documentclass[12pt]{amsart}
\usepackage{amsmath, amsfonts,amsthm,times,graphics}
\input amssym.def
\input amssym

\begin{document}

\newtheorem{theorem}{Theorem}[section]
\newtheorem{lemma}[theorem]{Lemma}
\newtheorem{corollary}[theorem]{Corollary}
\newtheorem{proposition}[theorem]{Proposition}
\newtheorem{conjecture}[theorem]{Conjecture}
\newtheorem{question}[theorem]{Question}
    \theoremstyle{definition}
\newtheorem{definition}[theorem]{Definition}
\newtheorem{example}[theorem]{Example}
\newtheorem{xca}[theorem]{Exercise}
\newtheorem{remark}[theorem]{{\it Remark}}

\newcommand{\abs}[1]{\lvert#1\rvert}

 \makeatletter

\title[On the distribution of primes \ldots]
{On the distribution of primes 
in the alternating sums of concecutive primes}

\author{Romeo Me\v strovi\' c}
\address{Maritime Faculty Kotor, University of Montenegro, Dobrota 36,
85330 Kotor, Montenegro} \email{romeo@ac.me}

\makeatother

{\renewcommand{\thefootnote}{}\footnote{2010 {\it Mathematics Subject 
Classification.} Primary 11A41, Secondary  11A25.

{\it Keywords and phrases}:  
alternating sum of primes, distribution of  primes,
 prime counting function,  Restricted Prime Number Theorem, 
Pillai's conjecture.}
\setcounter{footnote}{0}}

 \maketitle
 \begin{abstract} Quite recently, in \cite{me2} the authoor of this paper 
 considered the distribution of primes in the sequence  
$(S_n)$ whose $n$th term is defined as 
$S_n=\sum_{k=1}^{2n}p_k$, where  $p_k$ is the $k$th prime. 
Some heuristic arguments and the 
numerical evidence lead to the conjecture that 
 the primes are distributed among sequence $(S_n)$ in the same way that they 
are distributed among positive integers. More precisely,
Conjecture 3.3 in \cite{me2} asserts that 
$\pi_n\sim \frac{n}{\log n}$ as $n\to \infty$, where $\pi_n$ denotes the
 number of primes in the set $\{S_1,S_2,\ldots, S_n\}$.
Motivated by this, here we consider the  distribution of primes
in  aletrnating sums of first $2n$ primes, i.e., 
in the sequences  $(A_n)$ and $(T_n)$ defined by
  $A_n:=\sum_{i=1}^{2n}(-1)^ip_i$ and  $T_n:=A_n-2=\sum_{i=2}^{2n}(-1)^ip_i$  
($n=1,2,\ldots$).

Heuristic arguments and computational results 
suggest the conjecture that (Conjecture 2.5)  
  $$
\pi_{(A_k)}(A_n)\sim \pi_{(T_k)}(T_n)\sim \frac{2n}{\log n}
\quad {\rm as}\,\, n\to \infty,
  $$ 
where $\pi_{(A_k)}(A_n)$ (respectively, $\pi_{(T_k)}(T_n)$) denotes the
 number of primes in the set $\{A_1,A_2,\ldots, A_n\}$
(respectively,  $\{T_1,T_2,\ldots, T_n\}$).
Under Conjecture 2.5 and Pillai's conjecture, we establish two 
results concerning the expressions for the $k$th prime in the 
sequences $(A_n)$ and $(T_n)$. Furthermore, we propose some other 
related conjectures and we deduce some their consequences.
 \end{abstract}

  \section{Introduction, Motivation and Preliminaries}

Motivated by the notion of
 generalized prime system or $g$-{\it prime system}) 
${\mathcal G}$ introduced by A. Beurling in \cite{be}  
which generalizes the notion of 
primes and positive integers, in \cite[Section 1]{me2} 
it was considered a 
system described as follows.

Let ${\mathcal P}:=\{p_1,p_2,p_3,\ldots \}$ be the set of all primes
$2=p_1<p_2<p_3<\cdots$  
 and let 
${\mathcal N}$ 
be an increasing integer sequence $(a_k)_{k=1}^{\infty}$.

Let  $({\mathcal P}, {\mathcal N}:=(a_k)_{k=1}^{\infty})$
be a pair defined above. Then we  define its {\it 
counting function}  \cite[p. 3]{me2}
$N_{(a_k)}(x)$  $x\in [1,\infty))$ as 
  $$
N_{(a_k)}(x)=\#\{i:\,i\in\Bbb N\,\,
{\rm and}\,\, a_i\le x\}.
   $$
Furthermore, the  {\it prime counting function}
for $({\mathcal P}, {\mathcal N})$
is the function $x\mapsto \pi_{(a_k)}(x)$ defined on 
$[1,\infty)$  as
  \begin{equation}\label{(1)}
\pi_{(a_k)}(x)=\# \{q:\, q\in {\mathcal P}\,\,
{\rm and}\,\, q=a_i \,\,{\rm for\,\,some\,\,} i 
\,\,{\rm with\,\,} a_i\le x\}.
  \end{equation}
Some heuristic and computational results 
show that for many ``natural pairs''  
$({\mathcal P}, {\mathcal N}:=(a_k)_{k=1}^{\infty})$
the associated 
counting function  $N_{(a_k)}(x)$
satisfies certain asymptotic growth as $x\to\infty$ (see \cite[Section 2]{me2}). 
Notice that for each positive integer $n$ 
we define \cite[the equality (2) with ${\mathcal G}={\mathcal P}$]{me2}
   \begin{equation}\label{(2)}
\pi_{(a_k)}(a_n)=\# \{q:\, q\in {\mathcal P}\,\,
{\rm and}\,\, q=a_i \,\,{\rm for\,\,some\,\,} i 
\,\,{\rm with\,\,} 1\le i\le n\}.
  \end{equation}

Accordingly, we give the following definition \cite[Definition 1.1]{me2}.

 \begin{definition}
 Let $\Omega$ be a set of all   nonnegative continuous real functions
defined on $[1,+\infty)$ and let $(a_k)_{k=1}^{\infty}:=(a_k)$ be an increasing
sequence of positive integers. We say that $(a_k)$ 
satisfies  $\omega$-{\it Restricted Prime Number Theorem} 
if there exists the function $\omega_{(a_k)}=\omega\in\Omega$ such that 
the function $n\mapsto \pi_{(a_k)}(a_n)$ defined by (2) is asymptotically 
equivalent to $\omega(n)$ as $n\to \infty$.

In particular, if $\omega(x)\sim x/\log x$ as $x\to\infty$,
then we say that a sequence $(a_k)$ satisfies the 
{\it Restricted Prime Number Theorem} (RPNT). 
  \end{definition}

Notice that by the {\it Prime Number Theorem},
   $$
\lim_{x\to\infty}\frac{\pi(x)}{\frac{x}{\log x}}=1,
   $$   
where $\pi(x)$ is the prime counting function, i.e., 
$\pi(x)$  denotes the number of primes less than $x$.
For history,  see \cite{bd1} and \cite[p. 21]{me}.

Motivated by our recent paper \cite{me2} 
concerning the distribution of primes in the sequence 
$(S_n)$ with $S_n=\sum_{i=1}^{2n}p_i$ ($n=1,2,\ldots$), 
computations (Table 1), Pillai's conjecture (Conjecture 2.5) 
and some  heuristic arguments, in the following section we propose
the conjecture (Conjecture 2.5) on the distribution of 
primes in the sequences $(A_n)$ and $(A_n-2)$ with 
$A_n=\sum_{i=1}^{2n}(-1)^ip_i$. Namely, Conjecture 2.2 asserts that the 
number of primes in in the set $\{A_1,A_2,\ldots, A_n\}$ is 
$\sim \frac{2n}{\log n}$ as $n\to \infty$.  
Under Pillai's conjecture and Conjecture 2.5, 
we deduce  two consequences concerning the asymptotic expressions 
for the $k$th prime in the sequences $(A_n)$ and $(A_n-2)$.   
Some related conjectures and their corollaries are also presented.
Finally, by using computational results up to $n=5\cdot 10^8$,
 we propose some conjectures on the 
estimates of differences $A_n-p_n$.

\section{The distribution of primes in alternating sums $(A_n)$ 
and  $(A_n-2)$ with $A_n:=\sum_{i=1}^{2n}(-1)^ip_i$}

In this section we consider the distribution of primes in 
alternating sums of consecutive primes; namely, in 
the sequences $(A_n)$ and $(T_n)$ respectively  defined by 
   $$
 A_n=\sum_{i=1}^{2n}(-1)^ip_i, \quad n=1,2,\ldots,
  $$
and
  $$
T_n=\sum_{i=2}^{2n}(-1)^ip_i, \quad n=1,2,\ldots.
  $$
Notice that $T_n=A_n-2$ ($n=1,2,\ldots$) and here we present  conjectures and related results based on 
computational results concerning  the 
sequence $(A_n)$. Notice that  computational 
results  and heuristic arguments
related to the sequence  $(T_n)$ suggest the same conjectures and 
their consequences as these for the sequence $(A_n)$. 

For computational investigations of distribution of primes presented 
in Table 1, we   proceed  similarly as in 
\cite[Section 6]{me2}, where 
the analogous study is considered for 
the sequence $(S_n)$ with $S_n=\sum_{i=1}^{2n}p_i$ ($n=1,2,\ldots$).

  \begin{remark}
Note that $(A_n)$ is the sequence 
consisting  of terms of  Sloane's sequence A008347 \cite{sl}
(firstly introduced by N.J.A. Sloane and J.H. Conway) with 
even indices defined as $a_n=\sum_{i=0}^{n-1}(-1)^ip_{n-i}$ 
$(a_0=0,2,1,4,3,8,5,\ldots)$; namely, $A_n=a_{2n}$ for all $n=1,2,\ldots$.
Notice also that the  ``complement'' (with respect to $\Bbb N$) of 
Sloane's sequence A008347 is the sequence 
A226913 $(6,9,10,11,14,15,17,\ldots)$.
Furthermore, the sequence $(A_n)$ is also closely related to Sloane's sequence 
A131694-numbers $n$ such that $b_n:=\sum_{i=1}^n(-1)^ip_i$ is 
a prime $(1,4,6,8,10,12,18,\ldots)$.

Recall also that Sloane's sequence A066033 is defined  as  
$a_n=2+\sum_{i=2}^{n}(-1)^ip_i$ with $a_1=2$ $(2,5,0,7,-4,9,-8,11,-12,17,
\ldots)$; the sequence A136288 
defined as primes which are the absolute value of the alternating sum 
and the difference of the first $n$ primes $(2,3,5,7,13,19,29,53,61,\ldots)$
(cf. the sequences A163057-an alternating sum from the $n$th odd number 
up to the $n$th odd prime $(2,4,6,9,11,14,16,\ldots)$,
and the related sequences A163058-primes in  A163057 $(2,11,19,23,\ldots)$.
Notice also that the sequences A264834, A242188, A240860, A233809, 
A226743, A131196 and A131197  are closely related to the sequence
A008347. 
    \end{remark}

In 1982 D.A. Goldston \cite{go} has proved assuming the 
Riemann Hypothesis that 
    \begin{equation*}
\sum_{p_i<x\atop p_{i}-p_{i-1}\ge d}(p_{i}-p_{i-1})=
O\left(\frac{x\log x}{d}\right)
     \end{equation*}
uniformly for $d\ge 2$, which for $d=2$ putting $x=p_{2n}$
and $p_{2n}\sim 2n\log n$ immediately yields 
 \begin{equation*}
A_n:=\sum_{i=1}^{2n}(p_{i}-p_{i-1})=O(n\log^2 n).
     \end{equation*}
Assuming the Riemann Hypothesis, 
as a consequence of a conjecture posed in 2011 by  M. Wolf
\cite[Conjecture 1]{wo}, Wolf \cite[the asymptotic relation (41) 
of Section 4]{wo} noticed that 
     \begin{equation*}
\sum_{p_i<x\atop p_{i}-p_{i-1}\ge d}(p_{i}-p_{i-1})\sim 
x+\frac{d(d-1)}{2}\cdot \frac{x}{\log^2 x}+O\left(\frac{1}{\log^3 x}\right)
     \end{equation*}
which for $x$ so large that $\log x>d$  is indeed smaller than 
the above upper bound of Goldston. In particular, for $d=2$,  
$k=p_{2n}$ and $p_{2n}\sim 2n\log n$  the  
previous estimate  gives 
   \begin{equation*}
A_n:=\sum_{i=1}^{2n}(p_{i}-p_{i-1})\sim p_{2n}\sim 2n\log n.
     \end{equation*}
However, a computation shows that the above 
asymptotic relation is probably false, i.e., 
it is probably true with $n\log n$ instead of  
$2n\log n$ (i.e., with $p_n$ instead of $p_{2n}$) on the right hand side. 
This is in fact the  following conjecture  due to Pillai \cite[p. 84, 
Conjecture 34]{mo} (also cf. \cite[Comments of Joseph L.Pe in Sloane's 
sequence A008347]{sl}).

 \begin{conjecture}
If $k\in\Bbb N$, then 
   \begin{equation}\label{(3)}
\big|\sum_{i\le k}(-1)^{i-1}p_i\big|\sim \frac{p_k}{2}\quad 
as \,\, k\to\infty.
  \end{equation} 
   \end{conjecture}
\begin{corollary}
Under Pillai's  Conjecture {\rm 2.2} we have
  \begin{equation}\label{(4)}
A_n\sim T_n\sim n\log n\quad as\,\, n\to\infty. 
  \end{equation} 
More precisely, 
  \begin{equation}\label{(5)}
A_n\sim T_n\sim n\log n +n\log\log n -n +o(n) \quad as\quad n\to\infty. 
  \end{equation}
  \end{corollary}
  \begin{proof} 
Taking $k=2n$ and the well 
known asymptotic relation  $p_{2n}\sim 2n\log 2n\sim 2n\log n$  into (3)
(see, e.g., \cite{ms}), we immediately obtain (4). 

Furthermore, by Cipolla's formula \cite{ci} 
for the approximation to the $k$th prime,
  $$
p_k=k\log k+k\log\log k-k+o(k).
   $$
Taking the above expression with $k=2n$ into (3), we immediately obtain (5). 
  \end{proof} 
\begin{remark}
 Since 
 $$
A_n=p_{2n}-(p_{2n-1}-p_{2n-2})-\cdots -(p_3-p_2)-p_1,
  $$
we see that $A_n<p_{2n}$ for all $n=1,2,\ldots$.

Using the  asymptotic relation (4) and the fact that $A_n$ is 
an odd integer for all $n\in\Bbb N$, 
some heuristic arguments together with the Prime Number Theorem
 suggest  that the ``probability'' of  $A_n$ being a prime is $2/\log n$.
Consequently, there are 
$\sim 2n/\log n$ primes that belong 
to the set $\{A_1,A_2,\ldots ,A_n\}$ (of course, 
the same assertion holds for the sequence $(T_n)$). 
This together with computational results given in Table 1 
  leads to the following conjecture.
 \end{remark}

    \begin{conjecture}
Let  $(A_n)$ and $(T_n)$ be the  sequences
 for which $A_n=\sum_{i=1}^{2n}(-1)^ip_i$ and
$T_n=A_n-2=\sum_{i=2}^{2n}(-1)^ip_i$. 
Then in accordance to the notion of Definition $1.1$,  
  \begin{equation}\label{(6)} 
\omega_{(A_n)}(x)=\omega_{(T_n)}(x)\sim \frac{2x}{\log x}
\quad {\rm as}\,\, x\to\infty,
    \end{equation} 
or equivalently,
   \begin{equation}\label{(7)}\begin{split}
\pi_{(A_k)}(A_n)& =\# \{p:\, p\,\,{\rm is\,\,a\,\, prime\,\,
and}\,\, p=A_i \,\,{\rm for\,\,some\,\,} i \,\,{\rm with\,\,} 1\le i\le n\}\\
&\sim \frac{2n}{\log n}\quad {\rm as\,\,} n\to\infty\end{split}
  \end{equation}
and 
 \begin{equation}\label{(8)}\begin{split}
\pi_{(T_k)}(T_n)& =\# \{p:\, p\,\,{\rm is\,\,a\,\, prime\,\,
and}\,\, p=T_i \,\,{\rm for\,\,some\,\,} i \,\,{\rm with\,\,} 1\le i\le n\}\\
&\sim \frac{2n}{\log n}\quad {\rm as\,\,} n\to\infty.\end{split}
  \end{equation}
  \end{conjecture}

Notice that in all our   results  of this 
section (Theorem 2.8 and Corollaries) we assume the truth of Conjectures 2.2 and 2.5.

Observe  that $|\sum_{i=1}^{n}(-1)^ip_i|$ is equal to $A_{n/2}$ for 
even $n$, while  $|\sum_{i=1}^{n}(-1)^ip_i|$ is even for odd $n$.
This fact shows that Conjecture 2.5 is equivalent with the following 
one. 
 \vspace{2mm}


\noindent{\bf Conjecture 2.5.'}
{\it Let  $(a_n)$ be the sequence
defined as $a_n=|\sum_{i=1}^{n}(-1)^ip_i|$. Then 
  \begin{equation}\label{(9)} 
\omega_{(a_n)}(x)=\frac{x}{\log x}.
    \end{equation} 
In other words, the sequence $(a_n)$ satisfies the  Restricted 
Prime Number Theorem.}
\vspace{2mm}

As a direct application of Conjectures 2.2 and 2.5, we obtain 
the following $A_n$ ($T_n$)-analogue of Corollary 3.6 in \cite{me2} concerning 
the sequence $(S_n)$ with $S_n=\sum_{i=1}^{2n}p_i$.

   \begin{corollary}[The asymptotic expression for 
the $k$th prime in the sequences $(A_n)$ and $(T_n)$]
Let $r_k$ $(k=1,2,\ldots)$ be the $k$th  prime  in the 
sequence $(A_n)$ {\rm(} or  $(T_n)$ {\rm)}. Then 
  \begin{equation}\label{(10)}
r_k\sim \frac{k\log^2 k}{2}\quad {\rm as \,\,} k\to\infty.
   \end{equation}
 \end{corollary}
\begin{proof}
If  a pair $(k,m)$ satisfies $r_k=A_m$, then by (6) of Conjecture 2.5,
    $2m\sim k\log m$ as $k\to\infty$, and hence,
 $\log m\sim\log k$ as $k\to\infty$. The previous two asymptotic 
relations immediately give
   \begin{equation}\label{(11)}
m\log m\sim \frac{k\log^2 k}{2}\quad {\rm as \,\,}k\to\infty.
    \end{equation}
Since by (4) of Corollary 2.3, $r_k\sim m\log m$ as $k\to\infty$,
substituting this into (11), we immediately obtain (10). 
  \end{proof}

Notice that the Prime Number Theorem and Corollary 2.3 immediately
yield the following result. 

 \begin{corollary} 
Under   Conjecture {\rm 2.2} there holds 
 \begin{equation}\label{(12)} 
\pi(A_n)\sim \pi(T_n)\sim n
\quad as\,\, n\to\infty.
    \end{equation} 
where $\pi(x)$ is the prime counting function.
 \end{corollary}
Furthermore, we have the following $A_n$-analogue of Theorem 4.4  
of  \cite{me2} concerning 
the sequence $(S_n)$ with $S_n=\sum_{i=1}^{2n}p_i$.

   \begin{theorem}[The asymptotic expression for 
the $k$th prime in the sequence $(A_n)$]
Let $r_k$ be the $k$th  prime in the sequence
$(A_n)$ $(k=2,3,\ldots)$. Then under Conjectures {\rm 2.2}
and {\rm 2.5},  there exists a positive sequence $(R_k)$ 
such that $\lim_{k\to\infty}R_k=1$ and  
    \begin{equation}\label{(13)}
r_k= \frac{1}{2}R_k^3k\log k(\log k+\log\log k+2\log R_k).
   \end{equation}
  \end{theorem}
Proof of Theorem 2.8 is based on  the following result.
  \begin{lemma} 
Let $r_k=A_m$ be the $k$th prime in the sequence $(A_n)$. Then
under Conjectures $2.2$ and $2.5$, 
  \begin{equation}\label{(14)}
r_k\sim \frac{m\sqrt{2m}\log m}{\sqrt{k\log k}}\quad as\,\, k\to\infty.
  \end{equation}
\end{lemma}
  \begin{proof}[Proof of Lemma $2.9$]
From  the proof of Corollary 2.6 we see  that $2m\sim k\log k$  
 and $r_k\sim m\log m$ as $k\to\infty$. The previous two asymptotic relations
immediately imply (14).
  \end{proof}

 \begin{proof}[Proof of Theorem $2.8$]
Proof of Theorem 2.8 is based on Lemma 2.9. Since this proof 
is completely similar to those of  Theorem 4.4 of \cite{me2}, it can be 
omitted.
   \end{proof}

Computational results (see seventh column in Table 1) suggest the 
following conjecture (cf. Conjecture 4.6 of \cite{me2}).
 \begin{conjecture}
For each pair $(k,m)$ with $k\ge 1$ and $r_k=A_m$ we have
  \begin{equation}\label{(15)}
\lfloor k\log k \rfloor +1\le 2m,
  \end{equation}
or equivalently,
   \begin{equation}\label{(16)}
r_k\ge A_{\lfloor (k+1)/2\rfloor}. 
   \end{equation}
 \end{conjecture}
Consequently, we can obtain the following two corollaries (cf. Corollaries 
 4.7, 4.8 and their proofs from \cite{me2}).
  \begin{corollary} 
If the inequality  {\rm (15)} of Conjecture {\rm 2.10} is true, then 
for each $k\ge 1$ there holds
    \begin{equation}\label{(17)}
r_k> \frac{k}{2}\log k(\log k+\log\log k).
    \end{equation}
 \end{corollary} 
 \begin{corollary} 
If the inequality  {\rm (15)} of Conjecture {\rm 2.10} is true, then 
$R_k>1$ for each $k\ge 1$, where $(R_k)$ is the sequence  defined by 
{\rm (13)} in Theorem {\rm 2.8}.
 \end{corollary} 

 \begin{remark}
As observed above, $A_n=a_{2n}$ for all $n=1,2,\ldots$,
where $(a_n)$ is Sloane's sequence A008347 \cite{sl} 
defined as $a_n=\sum_{i=0}^{n-1}(-1)^ip_{n-i}$.
Z.-W. Sun \cite[Conjectures  (i)--(iv) in Comments of Sloane's sequence 
A008347]{sl} proposed certain conjectures involving the sequence A008347. 
In particular, Sun conjectured that for each $n>9$,
  \begin{equation}\label{(18)}
a_{n+1}<(a_{n-1})^{\left(1+2/(n+2)\right)}.
  \end{equation}
The conjecture has been verified by Sun for $n$ up to $10^8$. 
Notice that (18) with $2n-1$ instead of $n$ can be written as
   \begin{equation*}
A_n<(A_{n-1})^{\left(1+2/(2n+1)\right)},\quad n=6,7,\ldots.
  \end{equation*}
Sun also conjectured that $(a_n)$ contains infinitely many 
{\it Sophie Germain primes} (given as Sloane's sequence A005384 in
\cite{sl}), and that there are infinitely many positive integers $n$ 
such that $a_n-1$ and $a_n+1$ are twin primes.
   \end{remark}

Table 1 obtained via  {\tt Mathematica 9}  presents our computational results 
concerning the  number  of    ``alternating prime sums'' $r_k$ 
(under Conjecture 2.5) and related expression 
(the equality (13) of Theorem 2.8).
The value $k$  
in the second column of Table 1  presents the 
number of primes in  set ${\mathcal A}_n:=\{A_1,A_2,\ldots ,A_n\}$,
where $n$ is a corresponding value given in the first column of this table.
Hence, under notations of Section 1 and Conjecture 2.5,
     \begin{equation*}
k:=\pi_{(A_k)}(A_n)=\# \{p:\, p\,\,{\rm is\,\,a\,\, prime\,\,
and}\,\, p=A_i \,\,{\rm for\,\,some\,\,} i \,\,{\rm with\,\,} 1\le i\le n\}.
   \end{equation*}
The appropriate  value of the greatest prime $r_k$ in 
${\mathcal A}_n$ is given in the third column,  
 while  after the value of $r_k$ in the bracket it  is 
written the value $n-m$, where 
$m$ is the index such  that $q_k=A_m$. In the 
fourth column  we present the corresponding  values of $R_k$ 
obtained as  solutions of the equation  (13) in Theorem  2.8.
The fifth column of Table 1 presents the values $R_k^{(u)}:=A_n/(n\log n)$  
which are upper bounds of $R_k$. Notice that  weakly but for computational 
purposes more suitable upper bounds of $R_k$ than $R_k^{(u)}$, are given as   
$R_k^{(u')}:=A_k/(k\log k)$. 

 The values of seventh column suggest that 
Conjecture 2.2 is probably true, but we believe that 
the values of these column are close to 1 for large values 
$n\gg 5\cdot 10^8$.   
Notice also that the values 
in the last column of Table 1 suggest the truth of Lemma 2.9 and 
Conjecture 2.14.

For example, from Table 1 we see that 
  $r_{33}=A_{96}=563$,  $r_{15234}=A_{9992}=1379813$, 
 $r_{129447}=A_{999994}=16230881$  and 
$r_{9833766}=A_{99999972}=2111199529$.

\vfill\eject

{\bf Table} 1. Distribution of primes in the sequence $(A_n)$ in the 
range $1\le n\le 5\cdot 10^8$

\begin{center}
{\tiny
\begin{tabular}{cccccccc}\hline
$n$ & $k$ & $r_k$ with $(n-m)$ & $R_k$  
& $\displaystyle R_k^{(u)}:=\frac{A_n}{n\log n}$  &  
$\displaystyle\frac{A_n-p_n}{n\log\log n}$ &
 $\displaystyle\frac{k\log m}{2m}$   & 
$\displaystyle\frac{r_k\sqrt{k\log k}}{m\sqrt{2m}\log m}$\\\hline
10 & 6  & 29(1)  &  1.24290   & 1.43317   & 0.47960 &    0.73241  & 1.03381
 \\
$10^2$ & 33   & 563(4)  & 1.23573  & 1.29637 & 0.36670  & 0.78450 & 1.00799    \\
$10^3$ & 254  &  8807(1)  & 1.23573   & 1.27523 & 0.46051  & 0.87804  & 0.95632    \\
$10^4$ & 1982  & 113557(4)  & 1.15094   & 1.23334 & 0.39931 &  0.91307 & 0.92720    \\
$10^5$ & 15234  & 1379813(8) & 1.15567   & 1.19862 & 0.32845 &   0.87700  & 1.02665 \\
$10^6$ & 129447  & 16230881(6)  & 1.13701  & 1.17484 & 0.28379  & 0.89412 
& 1.02546   \\
$10^7$ & 1116732  & 186806173(11)   & 1.12667  & 1.15899 & 0.26554  & 0.89998 & 1.02000   \\
$2\cdot 10^7$ & 2144771  & 388274699(4)  & 1.12395  & 1.11113 & 0.26022   
& 0.90141  &  1.02100  \\
$5\cdot 10^7$ & 5097220  & 1019145103(2)  & 1.12042  & 1.10839  & 0.25525  
& 0.90361 &   1.02016 \\
$7\cdot 10^7$ & 7007444  & 1451570059(19)  & 1.12926   & 1.10742  & 0.25307  
& 0.90416 &  1.01966  \\
$10^8$ & 9822766  & 2111199529(28)  & 1.11807  & 1.14610 &   0.25099 &   
0.90471 & 1.01916   \\
$10^8+5\cdot 10^7$ & 14431395  & 3230666071(2)  & 1.11662   & 1.14404 
& 0.24854 & 0.90563  &  1.01877  \\
$2\cdot 10^8$ & 18966586  & 4368109771(8)  & 1.11561  & 1.14266  & 0.24721 & 
 0.90631 & 1.01858   \\
$3\cdot 10^8$ & 27883839  & 6680071639(1)  & 1.11427  & 1.14076  & 0.24538 & 
 0.90712 & 1.01823   \\
$4\cdot 10^8$ & 36664392  & 9027893009(0)  & 1.11332  & 1.13948  & 0.24436 & 
 0.90776 & 1.01807  \\
$5\cdot 10^8$ &  45345672 & 11401770283(28)  & 1.11126  & 1.13846  
& 0.24322 &  0.90828 & 1.01791  \\
 \end{tabular}}
 \end{center}

\vspace{2mm}

In view of the data  of the last column in  Table 1, 
we propose the following two conjectures (cf. Conjecture 4.9 
of \cite{me2} concerning the $k$th prime in the 
sequence $(S_n)$ with $S_n=\sum_{i=1}^{2n}p_i$).
 \begin{conjecture} For every 
 $k\ge 15234$  with $r_k=A_m$ there holds 
   \begin{equation*}
r_k>\frac{m\sqrt{2m}\log m}{\sqrt{k\log k}}.
   \end{equation*}
\end{conjecture}

Furthermore, heuristic arguments, some computational results
and  Conjecture 2.5 lead to the following its two  generalizations 
(cf. \cite[Conjectures 3.9 and 3.18]{me2}). 

    \begin{conjecture}
For  any fixed nonnegative integer $d$  
the sequence $(A_n^{(d)})_{n=1}^{\infty}$ defined as 
 $$
A_n^{(d)}=2d+A_n=2d+\sum_{i=1}^{2n}(-1)^ip_i,\quad n=1,2,\ldots
 $$ 
satisfies the Restricted Prime Number Theorem. In other words,  
as $n\to\infty$, 
 \begin{equation}\label{(19)}\begin{split}
&\pi_{(2d+A_k)}(2d+A_n):=\# \{p:\, p\,\,{\rm is\,\,a\,\, prime\,\,
and}\,\, p=2d+S_i \\
&{\rm for\,\,some\,\,} i \,\,{\rm with\,\,} 1\le i\le n\}
\sim \frac{2n}{\log n}.
\end{split}\end{equation}
    \end{conjecture}
Notice that For $d=-1$, Conjecture 2.15 is in fact the part of 
 Conjecture {\rm 2.5} concerning the sequence $(T_n)$ with  
$T_n=A_n-2$ ($n=1,2,\ldots$).

    \begin{conjecture}
For any fixed positive integer $k$, let 
$(A_n^{(k)}):=(A_n^{(k)})_{n=1}^{\infty}$ 
be the sequence whose $n$th term is 
defined as
       $$
A_n^{(k)}=\sum_{i=1}^{2n+1}(-1)^{i-1}p_{i+k},\quad n=1,2,\ldots.
       $$
Then the sequence  $(A_n^{(k)})$ satisfies the Restricted 
  Prime Number Theorem. 
  \end{conjecture}

Suppose that $a$ and $d$ are relatively prime positive integers. 
Then  Dirichlet's  theorem \cite{dir}
asserts that that there are infinitely many primes of the form 
$kd+a$ with $k\in\Bbb N\cup\{0\}$.  
Dirichlet's  theorem, Conjecture 2.5 and some computational results 
lead to the following conjecture.
    \begin{conjecture}[Dirichlet's  theorem for the sequences $A_n$ and $T_n$]
Suppose that  $a$ and $d$ are relatively 
prime positive integers. Then in the sequence $A_n$ ($T_n$) 
 there are infinitely many primes of the form 
$kd+a$ with $k\in\Bbb N\cup\{0\}$.
    \end{conjecture}
Finally, sixth column of Table 1 suggests the following conjecture.
    \begin{conjecture} For each $n\ge 50$,
  \begin{equation}\label{(20)}
A_n-p_n> \frac{n\log\log n}{5},
  \end{equation}
and for each 
$n\ge 10^8+5\cdot 10^7$,
  \begin{equation}\label{(21)}
A_n-p_n<\frac{n\log\log n}{4}.
  \end{equation}
  \end{conjecture}

Notice that  the  inequalitiy $p_n<n\log n+n\log\log n$ with $n\ge 6$
 (see \cite{du2} and \cite[(3.13) of Corollary, p.69]{rs1}) 
together with some  additional computations 
immediately yields the consequence of the  first part of  Conjecture 2.18 
given as follows.

\begin{corollary} Under the inequality $(20)$ of the first part of 
Conjecture $2.18$, we have 
  \begin{equation}\label{(22)}
\frac{A_n}{p_n}-1>\frac{\log\log n}{5\log n}\quad {\rm for\,\, each}\quad
n\ge 50.
\end{equation}
\end{corollary}

Similarly, the  inequalitiy $p_n>n\log n$
with $n\ge 1$ (see, e.g., \cite[(3.12) of Corollary, p.69]{rs1})
immediately yields the consequence of the  second part of  Conjecture 2.18 
given as follows.

\begin{corollary} Under the inequality $(21)$ of the second part of 
Conjecture $2.18$, we have 
  \begin{equation}\label{(23)}
\frac{A_n}{p_n}-1<\frac{\log\log n}{4\log n}\quad {\rm for\,\, each}\quad
n\ge 10^8+5\cdot 10^7.
\end{equation}
\end{corollary}

\begin{remark} 
From (22) and (23) it follows that for every 
$n\ge 10^8+5\cdot 10^7$  there exists a real number  $C_n$
with $1/5<C_n<1/4$ such that 
 \begin{equation}\label{(24)}
A_n-p_n=C_n n\log\log n.
  \end{equation}
Observe that the equality (24) is a refined version of ``even case'' of
Pillai's conjecture (i.e., Conjecture 2.2 for even positive integers 
$k$ such that $k\ge 2(10^8+5\cdot 10^7)$).  
\end{remark}

\end{document}